\documentclass[reqno,12pt]{amsart}

\usepackage{amsmath, amsthm, amsfonts, amssymb,color}
\input{mathrsfs.sty}

\usepackage{amsmath}
\usepackage{amsfonts}
\usepackage{amssymb}
\usepackage{eufrak}
\usepackage{amscd}
\usepackage{amsthm}
\usepackage{amstext}
\usepackage[all]{xy}

\newcommand{\id}{\operatorname{id}}

\newcommand{\supp}{\operatorname{supp}}

\newcommand{\diam}{\operatorname{diam}}
\newcommand{\St}{\operatorname{St}}

   \theoremstyle{plain}
   \newtheorem{thm}{Theorem}[section]
   
   \newtheorem{lem}[thm]{Lemma}
   \newtheorem{cor}[thm]{Corollary}
   \theoremstyle{definition}
   
   \newtheorem{defn}[thm]{Definition}

   \theoremstyle{remark}

 \numberwithin{equation}{section}

\author{V. Manuilov}

\date{}

\address{Moscow Center for Fundamental and Applied Mathematics, Moscow State University,
Leninskie Gory 1, Moscow, 
119991, Russia}

\email{manuilov@mech.math.msu.su}

\thanks{Supported by the RSF grant 23-21-00068.}

\title {A continuous field of Roe algebras}

\begin{document}

\begin{abstract}
Let $X$ be a metric measure space. A Delone subset $D\subset X$ is a uniformly discrete set coarsely equivalent to $X$. We consider the space $\mathcal D_F$ of controlled Delone subsets of $X$ with an appropriate metric, and show that it, together with $X$ itself, is a compact space. By assigning to each point $D$ of $\mathcal D_F$ (resp., to $X$) the uniform Roe algebra $C^*_u(D)$ (resp., the \u Spakula's version $C_k^*(X)$ of the Roe algebra of $X$) we get a tautological family of $C^*$-algebras. For a sequence $\{D_n\}_{n\in\mathbb N}$ of controlled Delone subsets convergent to $X$ we show that the corresponding uniform Roe algebras $C^*_u(D_n)$, together with $C^*_k(X)$, form a continuous field of $C^*$-algebras over $\mathbb N\cup\{\infty\}$ when $X$ is a proper metric measure space of bounded geometry with no isolated points.

\end{abstract}

\maketitle

\section*{Introduction}

Let $X$ be a metric measure space. A Delone subset $D\subset X$ is a uniformly discrete set coarsely equivalent to $X$. We consider the space $\mathcal D_F$ of controlled (see Definition \ref{controlled}) Delone subsets of $X$ with an appropriate metric, and show that it, together with $X$ itself, is a compact space. By assigning to each point $D$ of $\mathcal D_F$ (resp., to $X$) the uniform Roe algebra $C^*_u(D)$ (resp., the \u Spakula's version $C_k^*(X)$ of the Roe algebra of $X$) we get a tautological family of $C^*$-algebras. For a sequence $\{D_n\}_{n\in\mathbb N}$ of controlled Delone subsets convergent to $X$ we show that the corresponding uniform Roe algebras $C^*_u(D_n)$, together with $C^*_k(X)$, form a continuous field of $C^*$-algebras over $\mathbb N\cup\{\infty\}$ when $X$ is a proper metric measure space of bounded geometry with no isolated points. The case when $X$ was a simplicial complex and the Delone subsets were of special form was considered in \cite{Lobachevskii}. It was erroneously claimed in \cite{Lobachevskii} that the fiber over $X$ is the Roe algebra $C^*(X)$. Here we fix this error by replacing $C^*(X)$ by $C^*_k(X)$and generalize \cite{Lobachevskii} to more general metric spaces.

The key role in the construction of this continuous field of $C^*$-algebras is played by partitions of unity and by partitions of $X$ of special form related to discretizations of $X$.

In Section \ref{S1} we define controlled Delone sets and prove compactness of the space of controlled Delone subsets of $X$ with $X$ added. In Section \ref{S2} we provide basic information about various versions of Roe algebras. In Section \ref{S3} we construct partitions of unity of a special form. In Section \ref{S4} we show that these partitions of unity can be orthogonalized in a way that allows to construct maps $\alpha^D:C^*_u(D)\to C^*(X)$ and backwards. In Section \ref{S5} we construct a partition of $X$ generalizing Voronoi cells. In Section \ref{S6} we show that the union of ranges of $\alpha^D$, $D\in D_F$, is dense in $C^*_k(X)$. In Section \ref{S7} we prove our main result that the constructed family of Roe algebras is a continuous field of $C^*$-algebras.

\section{Space of controlled Delone sets}\label{S1}

Let $(X,d)$ be a metric space. 
A discrete subset $D\subset X$ is {\it $r$-discrete} if $d(x,y)\geq r$ for any $x,y\in D$, $x\neq y$. It is called {\it $R$-dense} if for any $x\in X$ there exists $y\in D$ with $d(x,y)\leq R$. If $D$ is both $r$-discrete and $R$-dense then it is called an {\it $(r,R)$-Delone} set. It is a {\it Delone} set if it is an $(r,R)$-Delone set for some $r,R>0$. Given a Delone set $D\subset X$, we write $r(D)=\inf_{x,y\in D,x\neq y}d(x,y)$ and $R(D)=\inf\{R: D \mbox{\ is\ }R\mbox{-dense}\}$.

By Zorn Lemma, Delone sets exist in any metric space $X$.

\begin{defn}\label{controlled}
A monotone continuous function $F:[0,1]\to[0,1]$ such that $F(0)=0$ is called a {\it control function}. A Delone set is {\it controlled} (by $F$) if $R(D)\geq R$ implies $r(D)\geq F(R)$.

\end{defn}

It makes sense to consider only such control functions $F$ that for any $0<R\leq 1$ there exists a Delone set $D$ with $R(D)\leq R$ and $r(D)\geq F(R)$.  

Let $\mathcal D_F$ denote the set of all Delone sets $D$ with $R(D)\leq 1$ that are controlled by a function $F$ satisfying (c1) and (c2). Clearly, if $F_1(t)\leq F_2(t)$ for any $t\in[0,\infty)$ then $\mathcal D_{F_2}\subset\mathcal D_{F_1}$. Set $\overline{\mathcal D}_F=\mathcal D_F\cup\{X\}$. 

The sets in $\overline{\mathcal D}_F$ are closed, and we can use the Chabauty metric on this subset of the set of all closed subsets of $X$. Various versions of this metric are scattered in the literature. We use the metric from \cite{Smilansky-Solomon}. Note that this metric is defined for the case $X=\mathbb R^n$ in \cite{Smilansky-Solomon}, but their proof of the triangle inequality works for any metric space. Let $x_0\in X$ be a fixed point, $B_r(x_0)$ the open ball of radius $r$ centered at $x_0$, and let $A^{(+\varepsilon)}$ denote the $\varepsilon$-neighborhood of the set $A\subset X$. The metric $\rho$ on $\overline{\mathcal D}_F$ is defined by 
$$
\rho(D_1,D_2)=\inf\left(\left\{\varepsilon>0:\begin{array}{cc}D_1\cap B_{\frac{1}{\varepsilon}}(x_0)\subset D_2^{(+\varepsilon)}\\D_2\cap B_{\frac{1}{\varepsilon}}(x_0)\subset D_1^{(+\varepsilon)} \end{array}\right\}\cup\{1\}\right).
$$

\begin{thm}
If $X$ is proper then $\overline{\mathcal D}_F$ is compact.

\end{thm}
\begin{proof}
First we show that the space of all closed subsets of a proper metric space $X$ is precompact, and then we show that $\overline{\mathcal D}_F$ is its closed subset. 

Precompactness is equivalent to existence of finite $\varepsilon$-dense subset for any $\varepsilon$. Fix $\varepsilon>0$. As $X$ is proper, $\overline{B}_{\frac{1}{\varepsilon}}(x_0)$ is compact, so there exists a finite $\varepsilon$-dense subset $A$ in $\overline{B}_{\frac{1}{\varepsilon}}(x_0)$. Let $\mathcal A$ be the finite set of all subsets of $A$. We claim that $\mathcal A$ is $\varepsilon$-dense in $\mathcal D_F$. Let $D\in\mathcal D_F$. Set $A_i=A\cap D^{(+\varepsilon)}$. Clearly, $A_i\subset D^{(+\varepsilon)}$. Let $x\in D$. As $A$ is $\varepsilon$-dense, there exists $a\in A$ such that $d(x,a)<\varepsilon$. Clearly, $a\in A_i$. As $d(x,a)<\varepsilon$, $x\in A_i^{(+\varepsilon)}$, hence $D\cap B_{\frac{1}{\varepsilon}}(x_0)\subset A_i^{(+\varepsilon)}$, hence $\rho(D,A_i)\leq\varepsilon$.

To show that $\overline{\mathcal D}_F$ is closed, consider a sequence $D_n\in\mathcal D_F$, $n\in\mathbb N$, such that it converges to some closed subset $E\subset X$, and suppose that $E\neq X$. Let $\rho(X,E)=2\alpha\neq 0$, then we may assume that $\rho(X,D_n)>\alpha$ for any $n\in\mathbb N$. 

Then for any $\beta<\alpha$ the inclusion $B_{\frac{1}{\beta}}(x_0)\subset D_n^{(+\beta)}$ does not hold, hence there exists $x_n\in X$ such that $d(x_n,D_n)\geq\beta$. This means that $R(D_n)\geq\beta$, hence $R(D_n)\geq\alpha$. Then $r(D_n)\geq F(\alpha)>0$. We claim that $E$ is $F(\alpha)$-discrete. Suppose the contrary, then there are $x,y\in E$ with $d(x,y)<R(\alpha)$. Let $d(x_0,x),d(x_0,y)<\frac{1}{\gamma}$, and let $\delta<\gamma$. There exists $n\in\mathbb N$ such that $\rho(E,D_n)<\delta$, which means that $E\cap B_{\frac{1}{\delta}}(x_0)\subset D_n^{(+\delta)}$, i.e. $x,y\in D_n^{(+\delta)}$, i.e. $d(x,y)\geq F(\alpha)-2\delta$. As $\delta>0$ was arbitrary, $E$ is $F(\alpha)$-discrete. Let $x\in X$. As $d(x,D_n)\leq R(D_n)\leq 1$ and as for each $\varepsilon>0$ there exists $n\in\mathbb N$ such that $D_n\cap B_{\frac{1}{\varepsilon}}\subset E^{(+\varepsilon)}$, we have $d(x,E)\leq 1+\varepsilon$. As $\varepsilon$ was arbitrary, we have $d(x,E)\leq 1$, and as $x\in X$ was arbitrary, we conclude that $E$ is 1-dense. 

Finally, we have to check that $E\in \mathcal D_F$. 
For $\delta>0$ find $x\in X$ such that $d(x,E)>R(E)-\delta$, and let $\varepsilon$ satisfy $x\in B_{\frac{1}{\varepsilon}}(x_0)$ and $\varepsilon\leq\delta$. 
Then find $n\in\mathbb N$ such that 
\begin{equation}\label{e1}
D_n\cap B_{\frac{1}{\varepsilon}}(x_0)\subset E^{(+\varepsilon)}
\end{equation}
and 
\begin{equation}\label{e2}
E\cap B_{\frac{1}{\varepsilon}}(x_0)\subset D_n^{(+\varepsilon)}. 
\end{equation}
It follows from (\ref{e2}) that if $x,y\in E$ then, for sufficiently small $\varepsilon$ we have $d(x,D_n)\leq\varepsilon$, $d(y,D_n\leq\varepsilon)$, hence $d(x,y)\geq r(D_n)-2\varepsilon$, i.e. 
\begin{equation}\label{e3}
r(E)\geq r(D_n)-2\varepsilon\geq r(D_n)-2\delta.
\end{equation} 

It follows from (\ref{e1}) that 
\begin{equation}\label{e4}
R(D_n)\geq d(x,D_n)>d(x,E)-\varepsilon>R(E)-\delta-\varepsilon\geq R(E)-2\delta.
\end{equation} 

As $F$ is monotone, we have $F(R(D_n))>F(R(E)-2\delta)$. Continuity of $F$ implies that for any $\varepsilon'>0$ there exists $\delta>0$ such that $F(R(E)-2\delta)>F(R(E))-\varepsilon'$. Then it follows from (\ref{e4}) that
\begin{equation}\label{e5}
F(R(E))<F(R(D_n))+\varepsilon'\leq r(D_n)+\varepsilon',
\end{equation}
and from (\ref{e3}) and (\ref{e5}) we have that
\begin{equation}\label{e6}
F(R(E))<r(D_n)+\varepsilon'\leq r(E)+\varepsilon'+2\varepsilon.
\end{equation}
Passing in (\ref{e6}) to the limit as $\varepsilon'\to 0$, we get
$$
F(R(E))\leq r(E).
$$
Thus, $E\in\mathcal D_F$.
\end{proof}

\begin{lem}\label{r-R}
Let $X$ has no isolated points, let $D_n\in\mathcal D_F$, $n\in\mathbb N$, and let $\lim_{n\to\infty}D_n=X$. Then $\lim_{n\to\infty}R(D_n)=0$.

\end{lem}
\begin{proof}
If $\lim_{n\to\infty}D_n=X$ then we claim that $\lim_{n\to\infty}r(D_n)=0$. Suppose the contrary. Then there exists $\alpha>0$ such that, passing to a subsequence, we have $r(D_n)\geq\alpha$ for any $n\in\mathbb N$. 
For each $n\in\mathbb N$, take $u_n\in D_n\cap B_2(x_0)$ (they exist as $R(D_n)\leq 1$). As $X$ is proper, there exists an accumulation point $x\in X$ for the sequence $\{u_n\}_{n\in\mathbb N}$, and once again we may pass to a convergent subsequence. Consider the two cases: 1) all but finitely many points in this subsequence equal $x$; 2) there exists a subsequence $\{u_n\}_{n\in\mathbb N}$ such that $\delta_n=d(x,u_n)$ is a decreasing sequence of strictly positive numbers. In the second case, find $m$ such that $\delta_m<\alpha/4$. If $n>m$ then $\delta_m-\delta_n\leq d(u_n,u_m)\leq \delta_m+\delta_n$, hence $d(u_n,u_m)<\alpha/2$. Then $d(u_m,v)\geq\alpha/2$ for any $v\in D_n$, $v\neq u_n$. Therefore 
$$
\rho(X,D_n)\geq d(u_m,D_n)\geq d(u_m,u_n)\geq \delta_m-\delta_n, 
$$
so if $\delta_n<\delta_m/2$ then $\rho(X,D_n)\geq \delta_m/2>0$ for sufficiently great $n$, and we get a contradiction. In the first case we use that $x$ is not isolated, hence there exists a sequence $\{x_m\}_{m\in\mathbb N}$ that is convergent to $x$. Let $m$ satisfy $0<d(x_m,x)<\alpha/2$. Then $\rho(X,D_n)\geq d(x_m,D_n)=d(x_m,x)$ for any $n\in\mathbb N$, which contradicts $\lim_{n\to\infty}D_n=X$.

Since $\lim_{n\to\infty}r(D_n)=0$, as all $D_n$ are controlled, we have also $\lim_{n\to\infty}R(D_n)=0$  
\end{proof}

\section{Preliminaries on  Roe algebras}\label{S2}

Recall the definition of the Roe algebra of a metric space $X$ \cite{Roe}.
Let $H_X$ be a Hilbert space with an action of the algebra $C_0(X)$ of continuous functions on $X$ vanishing at infinity (i.e. a $*$-homomorphism $\pi:C_0(X)\to\mathbb B(H_X)$). For a Hilbert space $H$ we write $\mathbb B(H)$ (resp., $\mathbb K(H)$) for the algebra of all bounded (resp., all compact) operators on $H$. 
We will assume that 
\begin{equation}\label{0}
\{\pi(f)\xi:f\in C_0(X),\xi\in H_X\} \quad\mbox{is\ dense\ in}\quad H_X 
\end{equation}
and that 
\begin{equation}\label{1}
\pi(f)\in\mathbb K(H_X)\quad \mbox{implies\ that}\quad f=0. 
\end{equation}
An operator $T\in\mathbb B(H_X)$ is {\it locally compact} if the operators $T\pi(f)$ and $\pi(f)T$ are compact for any $f\in C_0(X)$. It has {\it finite propagation} if there exists some $R>0$ such that $\pi(f)T\pi(g)=0$ whenever the distance between the supports of $f,g\in C_0(X)$ is greater than $R$. The {\it Roe algebra} $C^*(X,H_X)$ is the norm completion of the $*$-algebra of locally compact, finite propagation operators on $H_X$. As it does not depend on the choice of $H_X$ satisfying (\ref{0}) and (\ref{1}), it is usually denoted by $C^*(X)$. It is convenient to have a Borel measure on $X$. When the measure $\mu$ on $X$ has no atoms, the simplest choice $H_X=L^2(X)$ with the standard action of $C_0(X)$ satisfies the properties (\ref{0}) and (\ref{1}), so it suffices to use $L^2(X)$ to define the Roe algebra. 

When $X$ is discrete, the choice $H_X=l^2(X)$ does not satisfy the condition (\ref{1}). In order to fix this, one may take $H_X=l^2(X)\otimes H$ with an infinitedimensional Hilbert space $H$ to obtain the Roe algebra. But there is also another option: still to use $H_X=l^2(X)$. The resulting $C^*$-algebra is called the {\it uniform Roe algebra} of $X$, and is denoted by $C^*_u(X)$. This $C^*$-algebra differs from the Roe algebra $C^*(X)$. It is more tractable, but has less relations with elliptic theory.

There is one more $C^*$-algebra of this type, suggested by J. \u Spakula \cite{SpakulaJFA}. Recall that a family $\mathcal T$ of operators on $L^2(X)$ is {\it uniformly approximable} if,  for every $\varepsilon > 0$ there is an
$N \in\mathbb N$ such that for every $T\in\mathcal T$ there is an operator $K$ of rank $\leq N$ such that $\|T - K\| < \varepsilon$. $T$ is {\it uniformly locally compact} if, for every $R > 0$ the family 
$$
\{\pi(f)T, T \pi(f) : f \in C_0(X), \diam(\supp f) \leq R \mbox{\ and\ } \|f\|_\infty \leq 1\}
$$ 
is uniformly approximable. The product of two uniformly approximable operators of finite propagation is an operator of the same type, so the subset of such operators is a $*$-subalgebra in $C^*(X)$, and its closure is a $C^*$-algebra denoted in \cite{SpakulaJFA} by $C^*_k(X)$. 

Let $D\subset X$ be a Delone set, and let $\{V_u\}_{u\in D}$ be a partition of $X$, i.e. a family of open subsets of $X$ such that 
\begin{itemize}
\item[(v1)]
there exist $r,R>0$ such that $B_r(u)\subset V_u\subset B_R(u)$ for any $u\in D$;
\item[(v2)]
$X=\cup_{u\in D}\overline{V}_u$;
\item[(v3)]
$\mu(\overline{V}_u\cap \overline{V}_v)=0$ for any $u,v\in D$;
\item[(v4)]
$\dim L^2(V_u)=\infty$ for any $u\in D$.
\end{itemize}
If $X$ is a simplicial complex then one may take Voronoi cells as $\{V_u\}_{u\in D}$.  

It follows from (v2) and (v3) that $L^2(X)=\oplus_{u\in D}L^2(V_u)$, and an operator on $L^2(X)$ can be written as an operator-valued infinite matrix $T=(T_{uv})_{u,v\in D}$, where $T_{uv}$ is an operator from $L^2(V_v)$ to $L^2(V_u)$.

Fix isometric isomorphisms $\zeta_u:l^2(\mathbb N)\to L^2(V_u)$. They induce $*$-isomorphisms $\tilde\zeta_u:\mathbb B(l^2(\mathbb N))\to\mathbb B(L^2(V_u))$. Let $p_n\in\mathbb B(l^2(\mathbb N))$ denote the projection onto the first $n$ coordinates. An operator $T=(T_{uv})_{u,v\in D}\in C^*(X)$ is uniformly locally compact if for any $\varepsilon>0$ there exists $n\in\mathbb N$ such that $\|T_{uv}-\tilde\zeta_u(p_n)T_{uv}\tilde\zeta_v(p_n)\|<\varepsilon$ for any $u,v\in D$.

It was shown in \cite{SpakulaJFA} that $C^*_u(D)\otimes\mathbb K\cong C^*_k(X)$ for any Delone set $D\subset X$.



\section{Partitions of unity}\label{S3}

Let $(X,d,\mu)$ be a metric measure space, that is, $X$ is a metric space equipped with a measure $\mu$ defined on
the Borel $\sigma$-algebra defined by the metric topology on $X$. We also assume that $X$ is proper.

Recall that  discrete metric space is of {\it bounded geometry} if for any $R>0$ there exists $N\in\mathbb N$ such that the number of points in each ball $B_R(u)$, $u\in D$, does not exceed $N$.

\begin{defn}
We say that $(X,d,\mu)$ is {\it of bounded geometry} if for every $R>0$ there exist $c,C>0$ such that $c\leq \mu(B_R(x))\leq C$ for any $x\in X$. 

\end{defn}

A similar definition is suggested in \cite{Winkel}.
The following Lemma justifies this term.

\begin{lem}\label{boundedgeometry}
Let $(X,d,\mu)$ be a measure metric space of bounded geometry, and let $D\subset X$ be an $r$-discrete subset. Then $D$ is of bounded geometry.

\end{lem}
\begin{proof}
Suppose the contrary: for some $R>0$ and for every $N\in\mathbb N$ there exists $u(N)\in D$ such that the number of points in $B_R(u(N))$ is greater than $N$. Let $u_1,\ldots,u_N$ be such points. Then $B_{r/2}(u)\cap B_{r/2}(v)=\emptyset$ for any $u,v\in D$, $v\neq u$, and $\cup_{i=1}^N B_{r/2}(u_i)\subset B_{R+r/2}(u(N))$, hence $\sum_{i=1}^N\mu(B_{r/2}(u_i))\leq \mu(B_{R+r/2}(u(N)))$. Let $c_{r/2}$, $C_{r/2}$ and $C_{R+r/2}$, $C_{R+r/2}$ be the constants from the definition of bounded geometry for radii $r/2$ and $R+r/2$. Then we have 
$$
Nc_{r/2}\leq \sum_{i=1}^N\mu(B_{r/2}(u_i))\leq \mu(B_{R+r/2}(u(N)))\leq C_{R+r/2}
$$  
for any $N\in\mathbb N$ --- a contradiction.
\end{proof}

\begin{lem}\label{L-pu}
Let $D$ be a $(r,R)$-Delone set in a measure metric space $X$ of bounded geometry. Then there exists a set $\phi^D=\{\phi^D_u\}_{u\in D}$ of functions in $L^2(X)$ with the following properties:
\begin{itemize}
\item[(p1)]
$\supp\phi^D_u\subset B_{2R}(u)$ for any $u\in D$;
\item[(p2)]
$\phi^D_u(x)=1$ for any $x\in B_{r/6}(u)$;
\item[(p3)]
$\sum_{u\in D}\phi^D_u(x)=1$ for any $x\in X$.
\item[(p4)]
there exists $L>0$ such that $\phi^D_u$ is $L$-Lipschitz for any $u\in D$.

\end{itemize}

\end{lem}
\begin{proof}
For $u\in D$, set $W_u=B_{2R}(u)\setminus(\cup_{v\in D,v\neq u} \overline{B_{r/6}(v)})$. This gives an open cover for $X$. Indeed, if $d(x,D)\geq r/6$ and $d(x,u)\leq R$ then $x\notin B_{r/6}(v)$ for any $v\in D$, hence $x\in W_u$. If $d(x,D)<r/6$ then $d(x,u)<r/6$ for some $u\in D$, hence $d(x,v)\geq r/3$ for any $v\in D$, $v\neq u$, hence $x\in W_u$. As $D$ is of bounded geometry, the cover by balls $R_{2R}(u)$, $u\in D$, has finite multiplicity, therefore the cover $\{W_u\}_{u\in D}$ has finite multiplicity as well. Set $h^D_u(x)=d(x,X\setminus W_u)$, 
\begin{equation}\label{p_u}
\phi^D_u(x)=\frac{h^D_u(x)}{\sum_{v\in D}h^D_v(x)}. 
\end{equation}
Then the family $\{\phi^D_u\}_{u\in D}$ is a partition of unity, i.e. (p3) holds. As $W_u\subset B_{2R}(u)$, (p1) holds. If $x\in B_{r/6}(u)$ then $x\notin W_v$ for any $v\in D$, $v\neq u$, hence $\phi^D_v(x)=0$, and $(p2)$ holds.  

It remains to show (p4). We use the following result from \cite{Bell-Dran}: the Lipschitz constant for the functions (\ref{p_u}) is bounded from above by a constant depending on the multiplicity of the cover $\{W_u\}_{u\in D}$ (which we already know to be finite) and by the Lebesgue number of this cover, so we need to show only that the Lebesgue number for this cover is strictly positive, i.e. that there exists $a>0$ such that any ball of radius $a$ lies in one of the sets of this cover.  
Take a point $x\in X$. Let $u\in D$ be the nearest point to $x$, and let $d(x,u)=b$. Consider two cases. First, suppose that $b<\frac{5r}{12}$. Then $d(x,v)\geq \frac{7r}{12}$ for any $v\in D$ different from $u$, hence $B_{5r/12}\subset W_u$. Second, suppose that $b\geq\frac{5r}{12}$. As this is the distance from $x$ to the nearest point of $D$, $d(x,v)\geq b\geq \frac{5r}{12}$ for any $v\in D$ different from $u$. As $D$ is $R$-dense, $d(x,u)<R$. Set $c=\min(\frac{r}{4},R)$. Then $B_c(x)\subset B_{2R}(u)$, and as $\frac{5r}{12}=\frac{r}{4}+\frac{r}{6}$, $B_c(x)\cap B_{r/6}=\emptyset$, hence $B_c(x)\subset W_u$. Thus, the Lebesgue number is at least $\min(\frac{5r}{12},c)>0$.  
\end{proof}

Let $L^2(X)=L^2(X,\mu)$ be the Hilbert space of square-integrable functions with the inner product $\langle \phi,\psi\rangle=\int_X\overline{\phi}\psi\,d\mu$.

Set $g^D_{uv}=\langle\phi^D_u,\phi^D_v\rangle$. The matrix $G^D=(g^D_{uv})_{u,v\in D}$ can be considered as a matrix of an operator on $l^2(D)$ with respect to the standard basis $\{\delta_u\}_{u\in D}$. 

\begin{lem}\label{L-bi}
The matrix $G^D$ is the matrix of a bounded positive invertible operator on $l^2(D)$. 

\end{lem}
\begin{proof}
Write $v\sim u$ if $g_{uv}\neq 0$. Let $N=\max_{u\in D}\{\#\{v\in D:v\sim u\}\}$, and let $g=\sup_{u,v\in D}|g_{uv}|$. 
Let $\xi=\sum_{u\in D}\xi_u\delta_u$. Then $G^D(\xi)=\eta_u\delta_u$, where $\eta_u=\sum_{v\in D}g^D_{uv}\xi_v$, and
\begin{eqnarray*}
\|G^D(\xi)\|^2&=&\sum_{u\in D}\Bigl|\sum_{v\in D}g^D_{uv}\xi_v\Bigr|^2
\leq\sum_{u\in D}Ng^2\sum_{v\sim u}|\xi_v|^2\\
&\leq&\sum_{u\in D}N^2g^2|\xi_u|^2=N^2g\|\xi\|^2,
\end{eqnarray*}
thus $G^D$ is bounded. Clearly, $G^D$ is selfadjoint, and even positive.

Let $Y=\cup_{u\in D}B_{r/6}(u)\subset X$.
Write $\langle \phi,\psi\rangle=\langle \phi,\psi\rangle_1+\langle \phi,\psi\rangle_2$, where 
$$
\langle \phi,\psi\rangle_1=\int_{Y}\overline{\phi}\psi\,d\mu, 
\quad
\langle \phi,\psi\rangle_2=\int_{X\setminus Y}\overline{\phi}\psi\,d\mu,
$$
and let $g^i_{uv}=\langle\phi^D_u,\phi^D_v\rangle_i$, $i=1,2$, $G^i=(g^i_{uv})_{u,v\in D}$, $G^D=G^1+G^2$.

Clearly, $G^2$ is positive, and $G^1$ is diagonal with diagonal entries $g^1_{uu}=\mu(B_{r/3}(u))$, hence uniformly separated from 0, therefore, $G^D\geq G^1$ is invertible.
\end{proof}

\section{Isometry from $l^2(D)$ to $L^2(X)$}\label{S4}

Given a partition of unity $\phi^D$ satisfying the conditions of Lemma \ref{L-pu}, set $\xi_u=(G^D)^{1/2}(\delta_u)$. By Lemma \ref{L-bi}, $\{\xi_u\}_{u\in D}$ is a basis for $l^2(D)$. Define a linear map $U_D:l^2(D)\to L^2(X)$ by setting $U_D(\xi_u)=\phi^D_u$. 

As 
\begin{eqnarray*}
\langle \xi_u,\xi_v\rangle&=&\langle(G^D)^{1/2}\delta_u,(G^D)^{1/2}\delta_v\rangle=\langle \delta_u,G^D\delta_v\rangle\\
&=&\langle \delta_u,\sum_{w\in D}g^D_{wv}\delta_w\rangle=g^D_{uv}=\langle\phi^D_u,\phi^D_v\rangle,
\end{eqnarray*}
$U_D$ is an isometry. Set $\psi^D_u=U^D(\delta_u)$. Then $\{\psi^D_u\}_{u\in D}$ is an orthonormal set in $L^2(X)$.
Denote by $H^D$ the linear span of the functions $\{\psi^D_u\}_{u\in D}$ (or, equivalently, of the functions $\{\phi^D_u\}_{u\in D}$). Then $P_D=U_DU^*_D$ is the projection onto $H^D$ in $L^2(X)$.

\begin{lem}\label{stronglimit}
Let $D_n\in\mathcal D_F$, $\lim_{n\to\infty}D_n=X$.
Then the sequence of projections $P_{D_n}$ is strongly convergent to the identity operator on $L^2(X)$. 

\end{lem}
\begin{proof}
Let $f\in C_0(X)$ be a continuous function with compact support $K\subset X$. For any $\varepsilon>0$ there exists $\delta>0$ such that $d(x,y)<\delta$ implies $|f(x)-f(y)|<\varepsilon$. Set $g_n=\sum_{u\in D_n}f(u)\phi_u^{D_n}$. Note that $g_n\in H^{D_n}$. Let $x\in X$. If $n$ is great enough then $2R(D_n)<\delta$ by Lemma \ref{r-R}, and one has 
\begin{eqnarray*}
|f(x)-g_n(x)|&=&|f(x)-\sum_{u\in D_n}f(u)\phi_u^{D_n}(x)|\\
&=&|f(x)-\sum_{u\in D_n, d(u,x)<2R(D_n)}f(u)\phi_u^{D_n}(x)|\\
&=&|\sum_{u\in D_n, d(u,x)<2R(D_n)}(f(x)-f(u))\phi_u^{D_n}(x)|\\
&\leq&\sum_{u\in D_n, d(u,x)<2R(D_n)}|f(x)-f(u)|\phi_u^{D_n}(x)\\
&\leq&\varepsilon\sum_{u\in D_n, d(u,x)<2R(D_n)}\phi_u^{D_n}(x)\\
&\leq&\varepsilon\sum_{u\in D_n}\phi_u^{D_n}(x) \ \ \  = \ \ \varepsilon. 
\end{eqnarray*}
Therefore, $\lim_{n\to\infty}\sup_{x\in X}|f(x)-g_n(x)|=0$. Then $\|f-g_n\|_{L^2(X)}\leq \mu(K)\varepsilon$, hence $\lim_{n\to\infty}\|f(x)-g_n(x)\|_{L^2(X)}=0$. Therefore, $\|f-P_{D_n}f\|_{L^2(X)}\leq\|f-g_n\|_{L^2(X)}$ also vanishes as $n\to\infty$. As continuous functions with compact support are dense in $L^2(X)$, we conclude that the $*$-strong limit of $P_{D_n}$ is the identity operator. 
\end{proof}

Define $\alpha^{D,\phi^D}:C_u^*(D)\to\mathbb B(L^2(X))$ by $\alpha^{D,\phi^D}(T)=U_D T U_D^*$, and $\beta^{D,\phi^D}: C^*(X)\to\mathbb B(l^2(D))$ by $\beta^{D,\phi^D}(S)=U_D^* S U_D$. 

\begin{lem}
The range of $\alpha^{D,\phi^D}$ lies in $C^*(X)$; the range of $\beta^{D,\phi^D}$ lies in $C^*_u(D)$.

\end{lem}
This Lemma was proved in \cite{Lobachevskii} for special Delone sets obtained from triangulations, but the proof works for general ones as well. It is based on the fact that $G^D$ has finite propagation, hence $(G^D)^{\pm 1/2}$ is a limit of finite propagation operators. 

Thus, we have maps
$$
\alpha^{D,\phi^D}:C^*_u(D)\to C^*(X);\quad \beta^{D,\phi^D}:C^*(X)\to C^*_u(D).
$$
The first one is a $*$-homomorphism, while the second one is only a completely positive map.

Note that if $\phi^D$ and $(\phi')^D$ are two partitions of unity satisfying the conditions of Lemma \ref{L-pu} then $\alpha^{D,\phi^D}$ and $\alpha^{D,(\phi')^D}$ (resp., $\beta^{D,\phi^D}$ and $\beta^{D,(\phi')^D}$) are unitarily equivalent, and also homotopic (as the linear path of partitions of unity is a partition of unity). We shall skip the partition of unity from the notation and write these maps as $\alpha^D$ and $\beta^D$ when the particular partition of unity does not matter.

\section{Partition of $X$}\label{S5}

Here we construct a special partition of $X$. For $A\subset X$ and for $\delta>0$ we write $A^{(-\delta)}$ for the set $A^{(-\delta)}=\{x\in X:d(x,X\setminus A)>\delta\}$.

\begin{lem}
Let $X$ be a proper metric measure space of bounded geometry with no isolated points, and let $D\subset X$ be a Delone set. Then there exists a family $\{V_u\}_{u\in D}$ of open subsets such that 
\begin{itemize}
\item[(V1)]
$B_{r(D)/2}(u)\subset V_u\subset B_{2R(D)}(u)$ for any $u\in D$;
\item[(V2)]
$V_u\cap V_v=\emptyset$ for $u\neq v\in D$, and $X=\cup_{u\in D}\overline{V}_u$;
\item[(V3)]
for any $u\in D$ and any $\varepsilon>0$ there exists $\delta>0$ such that $\mu(V_u\setminus V_u^{(-\delta)})<\varepsilon$;
\item[(V4)]
$\dim L^2(V_u)=\infty$.

\end{itemize}

\end{lem}
\begin{proof}
Let $\mathcal A$ be the set of collections $V^\alpha=\{V^\alpha_u\}_{u\in D}$, $\alpha\in\mathcal A$, such that
\begin{itemize}
\item[(0)]
each $V_u^\alpha$ is an open subset of $X$;
\item[(1)]
$B_{r(D)/2}(u)\subset V_u\subset B_{2R(D)}(u)$ for any $u\in D$
\item[(2)]
$V_u\cap V_v=\emptyset$ for $u\neq v\in D$;
\item[(3)]
for any $u\in D$ and any $\varepsilon>0$ there exists $\delta>0$ such that $\mu(V_u\setminus V_u^{(-\delta)})<\varepsilon$.
\end{itemize} 

This is a partially ordered set: $V^\alpha\leq V^\beta$ if $V^\alpha_u\subset V^\beta_u$ for any $u\in D$.

Consider a chain $(V^\beta)_{\beta\in\mathcal B}$, $\mathcal B\subset\mathcal A$. We claim that $V$ given by $V_u=\cup_{\beta\in\mathcal B}V_u^\beta$ is the maximal element for this chain. Clearly, $V$ satisfies (0)--(2), so it remains to check (3). 

Fix $u\in D$, and note that, for any $\delta>0$ there exists $\beta\in\mathcal B$ such that $V_u^{(-\delta)}\subset V_u^\beta\subset V_u$. Indeed, suppose the contrary. Then for any $\beta\in\mathcal B$ there exists $x_\beta\in X$ such that $x_\beta\in V^{(-\delta)}_u$ but $x_\beta\notin V_u^\beta$. As $B_{2R(D)}(u)$ is precompact, there exists an accumulation point $x$ for $\{x_\beta\}_{\beta\in\mathcal B}$. As $x\in V_u^{(-\delta/2)}$, there exists $\beta_0\in\mathcal B$ such that $x\in V_u^\beta$ for any $\beta\geq \beta_0$. As $V_u^{\beta_0}$ is an open set, it (and all greater sets) contains a neighborhood of $x$, hence infinitely many $x_\beta$, $\beta\in\mathcal B$. This contradicts $x_\beta\notin V_u^\beta$.

Taking $\delta=\frac{1}{n}$, $n\in\mathbb N$, we may find $\beta_n\in\mathcal B$ such that $V_u^{(-\frac{1}{n})}\subset V_u^{\beta_n}\subset V_u$. Then $V_u=\cup_{n\in\mathbb N}V_u^{\beta_n}=V_u$. This implies, by sigma-additivity of the measure, that $\lim_{n\to\infty}\mu(V_u^{\beta_n})=\mu(V_u)$. 

Fix $\varepsilon>0$, and find $\beta\in\mathcal B$ such that $\mu(V_u\setminus V_u^\beta)<\frac{\varepsilon}{2}$. Then find $\delta$ such that $\mu(V_u^{\beta}\setminus (V_u^\beta)^{(-\delta)})<\frac{\varepsilon}{2}$. As $(V_u^\beta)^{(-\delta)}\subset V_u^{(-\delta)}$, we have $V_u\setminus V_u^{-\delta}\subset V_u\setminus (V_u^\beta)^{(-\delta)}$, therefore,
$$
\mu(V_u\setminus V_u^{-\delta})\leq\mu(V_u\setminus (V_u^\beta)^{(-\delta)})=\mu(V_u\setminus V_u^\beta)+\mu(V_u^\beta\setminus(V_u^\beta)^{(-\delta)})<\frac{\varepsilon}{2}+\frac{\varepsilon}{2}=\varepsilon.
$$
Thus, $V\in\mathcal A$. 

By Zorn Lemma, $\mathcal A$ has a maximal element. Denote it still by $V$. Clearly, $V$ satisfies (V1), (V3), and the first part of (V2). To see that it satifies the second part of (V2), note that if there exists a point $x\in X$ such that $x\notin\cup_{u\in D}\overline{V}_u$ then there exists $c>0$ such that $B_c(x)\cap(\cup_{u\in D}\overline{V}_u)=\emptyset$. Then we can find $u\in D$ such that $d(x,u)<R(D)$, and replace $V_u$ by $V_u\cup B_c(x)$, which is greater than $V_u$ --- a contradiction with maximality of $V$.

As $L^2(B_{r(D)/2}(u))\subset L^2(V_u)$, it suffices to show that $\dim L^2(B_{r(D)}(u))=\infty$. As $X$ has no isolated points, there exists a sequence $\{x_k\}_{k\in\mathbb N}$ such that $u=\lim_{k\to\infty}x_k$. For each $k\in N$ there exists $\varepsilon_k>0$ such that $B_{\varepsilon_k}(x_k)\cap\{x_j:j\in\mathbb N\}=\{x_k\}$. Then $\mu(B_{\varepsilon_k})>0$, hence $L^2(B_{\varepsilon_k})\neq 0$ for each $k\in N$, hence $L^2(B_{r(D)/2}(u))\supset\oplus_{k\in\mathbb N}L^2(B_{\varepsilon_k}(x_k))$ is infinitedimensional, and (V4) holds.
\end{proof}

Note that (V3) implies that $\mu(\overline{V}_u\cap\overline{V}_v)=0$ for any $u\neq v\in D$, and that  (V2)  implies that $L^2(X)=\oplus_{u\in D}L^2(V_u)$.

Fix $D=D_1$, let $\{V_u\}_{u\in D}$ be a partition of $X$ satisfying (V1)-(V3), and let $\{\phi_v\}_{v\in D_n}$ be the partition of unity contructed in Lemma \ref{L-pu}. Write $v\prec u$ if $\supp\phi_v^{D_n}\subset V_u$. Let $Q_u^{D_n}$ be the projection onto the span of functions $\phi_v^{D_n}$, $v\prec u$. Then the range of $Q_u^{D_n}$ lies in $L^2(V_u)$.

\begin{lem}
Let $\{D_n\}_{n\in\mathbb N}$ be a sequence of points in $\mathcal D_F$ convergent to $X$, and let $\{V_u\}_{u\in D_1}$ be a partition of $X$ satisfying (V1)-(V3). Then the sequence $\{Q_u^{D_n}\}_{n\in\mathbb N}$ is strongly convergent to the identity operator on $L^2(V_u)$ for any $u\in D_1$. 

\end{lem}
\begin{proof}
Let $u\in D_1$, $f$ a bounded continuous function on $V_u$,  $|f(x)|\leq M$, $x\in V_u$. Let $\delta_n=2R(D_n)$. By Lemma \ref{r-R}, $\lim_{n\to\infty}\delta_n=0$.  If $x\in V_u^{(-\delta_n)}$ and if $v\in D_n\cap X\setminus V_u^{(-\delta_n)}$ then $\phi_v^{D_n}(x)=0$. Therefore, as in the proof of Lemma \ref{stronglimit}, 
$$
\lim_{n\to\infty}\sup_{x\in V_u^{(-\delta_n)}}|f(x)-\sum_{v\prec u}f(v)\phi_v^{D_n}(x)|=0.
$$
If $x\in V_u\setminus V_u^{(-\delta_n)}$ then
$$
|f(x)-\sum_{v\prec u}f(v)\phi_v^{D_n}(x)|\leq 2M,
$$
hence $\|f-\sum_{v\in D_n}f(v)\phi_v^{D_n}\|_{L^2(V_u\setminus V_u^{(-\delta_n)})}\leq 2M\mu(V_u\setminus V_u^{(-\delta_n)})$, and it vanishes as $n\to\infty$.
Thus, as
\begin{eqnarray*}
\|f-\sum_{v\in D_n}f(v)\phi_v^{D_n}\|_{L^2(V_u)}&=&\|f-\sum_{v\in D_n}f(v)\phi_v^{D_n}\|_{L^2(V_u^{(-\delta_n)})}\\
&+&\|f-\sum_{v\in D_n}f(v)\phi_v^{D_n}\|_{L^2(V_u\setminus V_u^{(-\delta_n)})},
\end{eqnarray*}
and both summands in the right hand side vanish as $n\to\infty$, so the sum vanishes too. As $\sum_{v\prec u}f(v)\phi_v^{D_n}$ lies in the range of $Q_u^{D_n}$, 
$$
\|f-Q_u^{D_n}f\|_{L^2(V_u)}\leq \|f-\sum_{v\in D_n}f(v)\phi_v^{D_n}\|_{L^2(V_u)}.
$$

As bounded continuous functions are dense in $L^2(V_u)$, we have $\lim_{n\to\infty}\|f-Q_u^{D_n}f\|_{L^2(V_u)}=0$ for any $f\in L^2(V_u)$.
\end{proof}

\section{Union of ranges of $\alpha^D$}\label{S6}

Let $\{D_n\}_{n\in\mathbb N}$ be a sequence of points in $\mathcal D_F$ convergent to $X$, $D=D_1$, let $\{\phi_v^{D_n}\}_{v\in D_n}$ be a partition of unity constructed in Lemma \ref{L-pu}. For $u\in D$ let $H_u^n$ denote the closure of the span of the functions $\phi_v^{D_n}$, $v\prec u$ (i.e. $\supp\phi_v^{D_n}\subset V_u$).
Set $m^n_u=\dim H_u^n$, $m^n=\min\{m^n_u:u\in D\}$.

\begin{lem}\label{dim}
One has $m(n)=\sup_{u\in D}m^n_u<\infty$ for each $n\in\mathbb N$, and $\lim_{n\to\infty}m^n=\infty$.

\end{lem}
\begin{proof}
Clearly, $m^n_u=\#\{v\in D_n, v\prec u\}$. By (V1), there exists $C>0$ such that $V_u\subset B_C(u)$ for any $u\in D$.   
By Lemma \ref{boundedgeometry}, $D_n$ is of bounded geometry, hence the number of points of $D_n\cap B_C(u)$ is bounded uniformly in $u$, and the first claim follows. 

For the second claim, suppose that there exists $N\in\mathbb N$ such that $m^n\leq N$ for infinitely many $n$'s. As $\{D_n\}_{n\in\mathbb N}$ is convergent to $X$, we have $\lim_{n\to\infty}R(D_n)=0$, hence, by (V1), $\lim_{n\to\infty}\mu(B_{2R(D_n)}(v))=0$ uniformly in $v$. Note that $\cup_{v\prec u}B_{2R(D_n)}(v)\supset B_{C}(u)$, and there exists $c>0$ such that $\mu(B_C(u))>c$ for any $u\in D$. Therefore,
\begin{equation}\label{measures}
0<c<\mu(B_C(u))\leq \sum_{v\prec u}\mu(B_{2R(D_n)}(v))\leq N\sup_{v\in D_n}\mu(B_{2R(D_n)}(v)).
\end{equation}
As the right hand side in (\ref{measures}) vanishes as $n\to\infty$, we get a contradiction, i.e. the second claim holds true. 
\end{proof} 

\begin{cor}\label{Corollary-asymptotic}
Let $T\in C_k^*(X)$. Then $\lim_{n\to\infty}\alpha^{D_n}\beta^{D_n}(T)-T=0$.

\end{cor}
\begin{proof}
Let $m\in\mathbb N$, $P_m=\oplus_{u\in D}\tilde\zeta_u(p_m)$. If $T\in C^*_k(X)$ then $\lim_{m\to\infty} T-P_mTP_m=0$.

Let $Q_u^n$ denote the projection in $L^2(V_u)$ onto $H_u^n$, and let $\zeta_u:l^2(\mathbb Z)\to L^2(V_u)$ satisfy $\tilde\zeta_u(p_{m_u^n})=Q_u^n$. Set $Q^n=\oplus_{u\in D}Q^n_u$. Then 
$$
T\geq \alpha^{D_n}\beta^{D_n}(T)=Q^n\alpha^{D_n}(T)Q^n\geq P_{m^n}\alpha^{D_n}(T)P_{m^n}. 
$$
By Lemma \ref{dim}, $\lim_{n\to\infty}m^n=\infty$, so $\lim_{n\to\infty}T-P_{m^n}\alpha^{D_n}(T)P_{m^n}=0$, hence the claim follows.

\end{proof}

\begin{lem}
Let $\{D_n\}_{n\in\mathbb N}$ be a sequence of Delone sets in $\mathcal D_F$ convergent to $X$, and let $\{V_u\}_{u\in D_1}$ be a partition of $X$ satisfying (V1)-(V4). Then $\cup_{n\in\mathbb N}\alpha^{D_n}(C^*_u(D_n))$ is dense in $C^*_k(X)$.

\end{lem}
\begin{proof}
If $T\in C^*_u(D_n)$ then $\alpha^{D_n}(T)=Q^n\alpha^{D_n}(T)Q^n=P_{m^n}\alpha^{D_n}(T)P_{m^n}$, where $P_m=\oplus_{u\in D}\tilde\zeta_u(p_m)$, hence $\alpha^{D_n}(T)\in C^*_k(X)$.

To show that the union of the ranges of $\alpha^{D_n}$ is dense in $C^*_k(X)$, let $m\in\mathbb N$, and let $T\in C^*_k(X)$ satisfy $T_{uv}=\tilde\zeta_u(p_m)T_{uv}\tilde\zeta_v(p_m)$ for any $u,v\in D$, i.e. $P_mTP_m=T$, where $P_m=\oplus_{u\in D}\tilde\zeta_u(p_m)$. Such operators are dense in $C^*_k(X)$.
 
Set $S=\beta^{D_n}(T)\in C^*_u(D_n)$, then $\alpha^{D_n}(S)=\alpha^{D_n}(\beta^{D_n}(T))=Q^nTQ^n$. If $m^n>m$ then $m_u^n>m$ for any $u\in D$, therefore $Q_u^n\geq \tilde\zeta_u(p_m)$. Then we have $Q_u^nT_{uv}Q_v^n=\tilde\zeta_u(p_m)T_{uv}\tilde\zeta_v(p_m)=T_{uv}$, hence $Q^nTQ^n=P_mTP_m=T$ i.e. $\alpha^{D_n}(S)=T$. 
\end{proof}

\section{Continuous field of Roe algebras}\label{S7}

Consider the compact metric space $\overline{\mathcal D}_F$. Over each `finite' point $D\in \overline{\mathcal D}_F$ we have the uniform Roe algebra $C^*_u(D)$, and over the `infinite' point $X\in \overline{\mathcal D}_F$ we have the Roe algebra $C^*(X)$ and, also, its smaller version $C^*_k(X)$ by \u Spakula. 

We don't know how to organize this tautological family of Roe algebras into a continuous field of $C^*$-algebras over $\overline{\mathcal D}_F$, but this becomes possible if we restrict ourselves to any sequence in $\overline{\mathcal D}_F$ that converges to $X$.

Recall that a continuous field of $C^*$-algebras over a locally compact Hausdorff space $T$ is a triple $(T,A,\pi_t:A\to A_t)$, where $A$ and $A_t$, $t\in A$, are $C^*$-algebras, the $*$-homomorphisms $\pi_t$ are surjective, the family $\{\pi_t\}_{t\in T}$ is faithful, and the map $t\mapsto\|\pi_t(a)\|$ is continuous for any $a\in A$ \cite{Fell}.

Let $\{D_n\}_{n\in\mathbb N}$ be a sequence of points in $\mathcal D_F$ such that $\lim_{n\to\infty}D_n=X$.
Set $T=\{D_n:n\in\mathbb N\}\cup\{X\}$. We identify $T$ with $\mathbb N\cup\{\infty\}$ with the metric from $\overline{\mathcal D}_F$. Set $A_n=C^*_u(D_n)$, $A_\infty=C^*_k(X)$. 
Let $F(T)=\prod_{t\in T}A_t$ denote the $C^*$-algebra of families $a=(a_t)_{t\in T}$ such that $a_t\in A_t$ for each $t\in T$, and $\sup_{t\in T}\|a_t\|<\infty$.

Set
$$
J=\{(a_t)_{t\in T}\in F(T): a_\infty=\lim_{t\to\infty}\|a_t\|=0\} \cong\oplus_{t\in\mathbb N}A_t.
$$ 
It is an ideal in $F(T)$.

For $S\in C^*(X)$ set 
$$
b_S(t)=\left\lbrace\begin{array}{cl}\beta^{D_t}(S)&\mbox{if\ }t\neq\infty;\\
S&\mbox{if\ }t=\infty.\end{array}\right.
$$
As $\beta^{D_t}$ is a contraction for any $t\in\mathbb N$, $b_S\in F(T)$.
Let $B$ be a $*$-subspace in $F(T)$ generated by $b_S$, $S\in C^*_k(X)$. 

Set $A=J+B\subset F(T)$.

\begin{lem}\label{Lem:product}
Let $R,S\in C^*_k(X)$. Then $\lim_{n\to\infty}\beta^{D_n}(RS)-\beta^{D_n}(R)\beta^{D_n}(S)=0$, i.e. $b_{RS}-b_R b_S\in J$.

\end{lem}
\begin{proof}
This follows directly from Corollary \ref{Corollary-asymptotic}.

\end{proof}

\begin{cor}
The linear space $A$ is a $*$-algebra.

\end{cor}

\begin{lem}
For any $S\in C^*(X)$ one has
$\lim_{n\to\infty}\|\beta^{D_n}(S)\|=\|S\|$.

\end{lem}
\begin{proof}
Note that $\|\beta^{D_n}(S)\|=\|P_{D_n}SP_{D_n}\|$. As $\|P_{D_n}\|=1$, the sequence $\|P_{D_n}SP_{D_n}\|$ is non-decreasing, and $\|P_{D_n}SP_{D_n}\|\leq \|S\|$. On the other hand, by Lemma \ref{stronglimit}, the strong limit of $P_{D_n}SP_{D_n}$ equals $S$. As the norm is lower semicontinuous, 
$\|S\|\leq\lim_{n\to\infty}\|P_{D_n}SP_{D_n}\|$.
\end{proof}

\begin{cor}
For any $b=(b_t)_{t\in T}\in A$ one has
\begin{equation}\label{120}
\lim_{n\to\infty}\|b_t\|=\|b_\infty\|.
\end{equation}

\end{cor}

\begin{lem}
The $*$-algebra $A$ is norm closed.

\end{lem}
\begin{proof}
Consider the norm closure $\overline A$ of $A$ in $F(T)$. Then $J$ is a closed ideal in $\overline A$. Let $\{a^i+b^i\}_{i\in\mathbb N}$, $a^i\in J$, $b^i\in B$, be a Cauchy sequence in $A$. Passing to the quotient $C^*$-algebra $\overline A/J$, the sequence $\{a^i+b^i+J\}=\{b^i+J\}$ is also a Cauchy sequence, as the quotient $*$-homomorphisms have norm 1. Note that
\begin{eqnarray*}
\|b+J\|=\inf_{a\in J}\|a+b\|\geq\lim_{n\to \infty}\|a_n+b_n\|=\lim_{n\to \infty}\|b\|=\|b_\infty\|,
\end{eqnarray*} 
hence $\{(b_\infty)^i\}_{i\in\mathbb N}$ is also a Cauchy sequence. As $A_\infty$ is norm closed, this sequence has a limit in $A_\infty$. But then $\{a^i\}_{i\in\mathbb N}$ is also a Cauchy sequence, and as $J$ is norm closed, its limit lies in $J$. Thus $\{a^i+b^i\}_{i\in\mathbb N}$ converges in $A$, hence $\overline A=A$.
\end{proof}

Let a $*$-homomorphism $\pi_t:A\to A_t$ be determined by $\pi_t(a+\beta(S))=a_t+\beta_t(S)$, $t\in T$. These maps are well defined as $J+B$ is a direct sum.

\begin{thm}
Let $X$ be a proper metric measure space of bounded geometry with no isolated points, and let $D_n$, $n\in\mathbb N$, be a sequence of controlled Delone sets convergent to $X$. Then the constructed above triple $(T,A,\pi_t:A\to A_t)$ is a continuous field of $C^*$-algebras.

\end{thm}
\begin{proof}
Each $\pi_t$ is clearly surjective. If $a\neq a'\in A$ then there exists $t\in T$ (finite or infinite) such that $a_t\neq a'_t$, hence the family $\{\pi_t\}_{t\in T}$ is faithful. Finally, we have to check that the map $t\mapsto \|\pi_t(a)\|$ is continuous at $t=\infty$, which follows from (\ref{120}).  
\end{proof}


\end{document}